\documentclass[12pt]{amsart}

\usepackage{amssymb}
\usepackage{bm}
\usepackage{graphicx}
\usepackage{psfrag}
\usepackage[usenames,dvipsnames]{xcolor}
\usepackage{enumerate}
\usepackage{multirow}
\usepackage{url}

\usepackage{comment}

\usepackage{algpseudocode}

\usepackage{mathtools}

\usepackage{xy}
\input xy
\xyoption{all}

\newcommand{\excise}[1]{}

\newtheorem{thm}{Theorem}[section]
\newtheorem{lemma}[thm]{Lemma}

\newtheorem{cor}[thm]{Corollary}

\newtheorem{prob}[thm]{Problem}

\newtheorem*{mainprob}{Main Problem}

\theoremstyle{definition}

\newtheorem{remark}[thm]{Remark}
\newtheorem{defn}[thm]{Definition}

\numberwithin{equation}{section}

\newcommand\NN{\mathbb{N}}

\newcommand\ZZ{\mathbb{Z}}

\newcommand\abs[1]{|#1|}

\newcommand\nsg{\mathcal{S}}
\newcommand\nsgtwo{\mathcal{T}}

\def\vec#1{\mathchoice{\mbox{\boldmath$\displaystyle\bf#1$}}
{\mbox{\boldmath$\textstyle\bf#1$}}
{\mbox{\boldmath$\scriptstyle\bf#1$}}
{\mbox{\boldmath$\scriptscriptstyle\bf#1$}}}

\begin{document}

\mbox{}
\title[When is a numerical semigroup a quotient?]{When is a numerical semigroup a quotient?}

\author[Bogart]{Tristram Bogart}
\address{Departamento de Matem\'aticas \\ Universidad de los Andes \\ Bogot\'a, Colombia}
\email{tc.bogart22@uniandes.edu.co}

\author[O'Neill]{Christopher O'Neill}
\address{Mathematics Department\\San Diego State University\\San Diego, CA 92182}
\email{cdoneill@sdsu.edu}

\author[Woods]{Kevin Woods}
\address{Department of Mathematics\\Oberlin College\\Oberlin, OH 44074}
\email{kwoods@oberlin.edu}

\date{\today}

\begin{abstract}
A natural operation on numerical semigroups is taking a quotient by a positive integer. If $\nsg$ is a quotient of a numerical semigroup with $k$ generators, we call $\nsg$ a $k$-quotient.  We give a necessary condition for a given numerical semigroup $\nsg$ to be a $k$-quotient, and present, for each $k \ge 3$, the first known family of numerical semigroups that cannot be written as a $k$-quotient.  We also examine the probability that a randomly selected numerical semigroup with $k$ generators is a $k$-quotient.  
\end{abstract}

\maketitle


\section{Introduction}
\label{sec:intro}

We denote $\NN=\{0,1,2,\dots\}$, and we define a \emph{numerical semigroup} to be a set $\nsg\subseteq\NN$ that is closed under addition and contains~0. A numerical semigroup can be defined by a set of generators,
\[\langle a_1,\ldots,a_n\rangle = \{a_1x_1+\cdots a_nx_n:\ x_i\in\NN\},\]
and if $a_1,\ldots,a_n$ are the minimal set of generators of $\nsg$, we say that $\nsg$ has \emph{embedding dimension} $\mathsf e(\nsg) = n$.  For example,
\[\langle 3,5\rangle = \{0,3,5,6,8,9,10,\ldots\}\]
has embedding dimension 2.  

If $\nsg$ is a numerical semigroup, then an interesting way to create a new numerical semigroup 
is by taking the \emph{quotient}
\[
\frac{\nsg}{d} = \{ t \in \NN:\  dt \in \nsg\}
\]
by some positive integer $d$.  Note that $\frac{1}{d}\nsg$ is itself a numerical semigroup, one that in particular satisfies $\nsg \subseteq \frac{1}{d}\nsg \subseteq \NN$.  For example, 
\[
\frac{\langle 3,5\rangle}{2}=\{0,3,4,5,\ldots\}=\langle 3,4,5\rangle.
\]
Quotients of numerical semigroups appear through the literature over the past couple of decades~\cite{symmetriconeelement,symmetricquotient} as well as recently~\cite{harrisquotient,nsquotientgens}; see~\cite[Chapter~5]{numerical} for a thorough overview.

\begin{defn}\label{def:quotientrank}
We say a numerical semigroup $\nsg$ is a \emph{$k$-quotient} if $\nsg=\langle a_1,\ldots,a_k\rangle/d$ for some positive integers $d, a_1, \ldots, a_k$.  The \emph{quotient rank} of $\nsg$ is the smallest $k$ such that $\nsg$ is a $k$-quotient, and we say $\nsg$ has \emph{full quotient rank} if its quotient rank is $\mathsf e(\nsg)$ (since $\nsg=\frac{\nsg}{1}$, its quotient rank is at most $\mathsf e(S)$).
\end{defn}


Numerical semigroups of quotient rank 2 are precisely the \emph{proportionally modular} numerical semigroups~\cite{openmodularns}, which have been well-studied~\cite{propmodtree,propmodular}.  
This includes arithmetical numerical semigroups (whose generators have the form $a, a + d, \ldots, a + kd$ with $\gcd(a,d) = 1$), which have a rich history in the numerical semigroup literature~\cite{diophantinefrob,setoflengthsets,nsfreeresarith}.  In fact, generalized arithmetical numerical semigroups~\cite{omidalirahmati}, whose generating sets have the form $a, ah + d, \ldots, ah + kd$, can also be shown to have quotient rank 3.  

For quotient rank $k \ge 3$, much less is known.  It is identified as an open problem in~\cite{nsgproblems} that no numerical semigroup had been proven to have quotient rank at least 4.  
Since then, the only progress in this direction is~\cite{ksquashed}, wherein it is shown there exist infinitely many numerical semigroups with quotient rank at least 4, though no explicit examples are given.  

With this in mind, we state the main question of the present paper.  

\begin{mainprob}\label{mainprob:whenaquotient}
When is a given numerical semigroup $\nsg$ a $k$-quotient?  
\end{mainprob}

Our main structural results, which are stated in Section~\ref{sec:necessary}, are as follows.  
\begin{itemize}
\item 
We prove a sufficient condition for full quotient rank (Theorem~\ref{thm:necessary}), which we use to obtain, for each $k$, a numerical semigroup of embedding dimension $k + 1$ that is not a $k$-quotient (Theorem~\ref{thm:noquotient}).  When $k \ge 3$, this is the first known example of a numerical semigroup that is not a $k$-quotient.  We~also construct, for each $k$, a numerical semigroup that cannot be written as an intersection of $k$-quotients (Theorem~\ref{thm:nointersection}), settling a conjecture posed in~\cite{ksquashed}.  

\item 
We prove quotient rank is sub-additive whenever the denominators are coprime.  
This provides a new method of proving a given numerical semigroup is a quotient:\ partition its generating set, and prove that each subset generates a quotient, e.g.,
\begin{align*}
\langle 11,12,13,17,18,19,20 \rangle
&= \langle 11,12,13 \rangle + \langle 17,18,19,20 \rangle
= \frac{\langle 11,13\rangle}{2}+\frac{\langle 17,20\rangle}{3} 
\\
&= \frac{3\langle 11,13\rangle + 2\langle 17,20\rangle}{2 \cdot 3}
= \frac{\langle 33,34,39,40\rangle}{6}.
\end{align*}
We~use this result to prove that any numerical semgiroup with \emph{maximal embedding dimension} (that is, the smallest generator equals the embedding dimension) fails to have full quotient rank (Theorem~\ref{thm:maxembdim}).  
\end{itemize}

Our remaining results are probabilistic in nature.  We examine two well-studied models for ``randomly selecting'' a numerical semigroup:\ the ``box'' model, where the number of generators and a bound on the generators are fixed~\cite{expectedfrob,arnoldfrob,burgeinsinaifrob}; as well as a model where the smallest generator and the number of gaps are fixed~\cite{kaplancounting}, whose prior study has yielded connections to enumerative combinatorics~\cite{kunzcoords} and polyhedral geometry~\cite{kunzfaces1,kunz}.  
We prove that under the first model, asymptotically all semigroups have full quotient rank (Theorem~\ref{thm:numericalbox}), while under the second model, asymptotically no semigroups have full quotient rank (Theorem~\ref{thm:maxembdim}).

Our results also represent partial progress on the following question, which has proved difficult.

\begin{prob}\label{prob:algorithm}
Given a numerical semigroup $\nsg$ and a positive number $k$, is there an algorithm to determine whether $\nsg$ is a $k$-quotient?  
\end{prob}

\begin{remark}\label{rem:relprime}
Some texts require that the generators of a numerical semigroup be relatively prime, so that $\NN\setminus \nsg$ is finite.  This assumption is harmless, since any numerical semigroup can be written as $m\nsg$, where the generators of $\nsg$ are relatively prime, and it also doesn't affect $k$-quotientability: given a positive integer $d$, one can readily check that
\[
\frac{m\nsg}{d} = m' \left(\frac{\nsg}{d'}\right),
\]
where $m' = m/\gcd(m,d)$ and $d' = d/\gcd(m,d)$.  
\end{remark}




\section{When is $\nsg$ not a $k$-quotient?}
\label{sec:necessary}

In this section, we give two structural results.  
The first (Theorem~\ref{thm:necessary}) is a necessary condition for a given numerical semigroup $\nsg$ to be a $k$-quotient, which forms the backbone of the constructions in Section~\ref{sec:fullquotientrank} and the probabilistic results in Section~\ref{sec:randomsgps}.  The second (Theorem~\ref{thm:sums}) is a constructive proof that quotient rank is sub-additive, provided the denominators are relatively prime.  

In what follows, we write $[p] = \{1,2,\dots,p\}$ for any positive integer $p$, and given a collection of vectors $\{\vec v_i\}$ and a set of indices $I$, we define $\vec v_I=\sum_{i\in I}\vec v_i$.  

\begin{thm} \label{thm:necessary}
Suppose
\[\nsg =\frac{\langle b_1,\ldots,b_k\rangle}{d}\]
for some $b_i\in \NN$ and positive integer $d$. Given any elements $s_1,\ldots,s_p \in \nsg$ with $p > k$, there exists a nonempty subset $I\subseteq [p]$ such that $s_I/2\in \nsg$.
\end{thm}

\begin{proof}
Let $\vec b=(b_1,\ldots,b_k)$. For $1\le i\le p$, let $\vec c_i=(c_{i1},\ldots,c_{ik})\in\NN^k$ be such that
\[s_i = d (c_{i1}b_1+\cdots c_{ik}b_k),\]
which exist since $s_i\in\nsg$.
For a vector $\vec v\in\ZZ^k$, define $\vec v\bmod 2\in\ZZ_2^k$ to be the coordinate-wise reduction of $\vec v$ modulo 2. For $J\subseteq [p]$, examine $\vec c_J \bmod 2$. There are $2^p$ possible $J$ and $2^k$ possible values for $\vec c_J \bmod 2$, with $p>k$, so there must be two distinct $J_1$ and $J_2$ such that
\[\vec c_{J_1}\bmod 2=\vec c_{J_2}\bmod 2.\]
Let $I=(J_1\setminus J_2)\cup (J_2\setminus J_1)$ be their symmetric difference, which is nonempty. Then
\[\vec c_{I}\bmod 2=\vec c_{J_1}+\vec c_{J_2}-2\vec c_{J_1\cap J_2}\bmod 2=\vec 0,\]
so $\vec c_{I}$ has even coordinates. Let $\vec c_I = (2q_1,\ldots,2q_k)$ where $q_i\in\NN$. Then
\begin{align*}
s_I/2&=\sum_{i\in I}\big(d\cdot (c_{i1}b_1+\cdots c_{ik}b_k)\big)/2
= d\sum_{j=1}^kb_j\sum_{i\in I}c_{ij}/2
=d\sum_{j=1}^k q_jb_j
\end{align*}
is an element of $\nsg$, as desired.
\end{proof}

\begin{cor} \label{cor:necessary}
Let $\nsg = \langle a_1, \dots, a_n \rangle$ be a numerical semigroup.  If $\nsg$ does not have full quotient rank, then there exists $I \subseteq [n]$ such that
\[ a_I \in\langle a_j:\ j\notin I \rangle.\]
\end{cor}

\begin{proof}
By applying Theorem~\ref{thm:necessary} to the generating set $\{a_1, \dots, a_n\}$, we obtain that for some $J \subseteq [n]$, $a_J/2 \in \nsg$. So there exist $c_r\in\NN$ such that 
\[
\sum_{j\in J} a_j=\sum_{r\in R} 2c_r a_r
\]
where $R=\{r:\ c_r>0\}$.  
Letting $I=J\setminus R$ and subtracting each $a_j$ with $j\in J\cap R$ from both sides, we have
\[
a_I=\sum_{i\in I} a_i = \sum_{r\in J\cap R} (2c_r-1) a_r + \sum_{r\in R\setminus J}2c_r a_r
\]
is an element of $\langle a_j:\ j\notin I\rangle$,
as desired. Note that $I$ is nonempty, as otherwise 
\[
0 = a_I = \sum_{r\in J\cap R} (2c_r-1) a_r + \sum_{r\in R\setminus J}2c_r a_r \ge \sum_{r\in J\cap R} a_r = \sum_{r\in J} a_r > 0
\]
since $J$ is nonempty, 
which is a contradiction.
\end{proof}


\begin{thm} \label{thm:sums}
If $\nsg$ and $\nsgtwo$ are numerical semigroups and $\gcd(c,d) = 1$, then
\[
\frac{\nsg}{c} + \frac{\nsgtwo}{d} = \frac{d\nsg + c \nsgtwo}{cd}.
\]
\end{thm}

\begin{proof}
First suppose that $x \in \frac{1}{c}\nsg + \frac{1}{d}\nsgtwo$.  Then $x = s+t$ where $cs \in \nsg$ and $dt \in \nsgtwo$, so
\[
cdx = d(cs) + c(dt) \in d \nsg + c \nsgtwo
\]
which implies $x \in \frac{1}{cd}(d\nsg + c \nsgtwo)$.  Note this containment does not require $\gcd(c,d) = 1$.  

On the other hand, suppose $cdx \in d \nsg + c \nsgtwo$, so
\begin{equation} \label{eq:sums}
cdx = ds + ct 
\qquad \text{for some} \qquad
s \in \nsg, t \in \nsgtwo.
\end{equation}
In particular, $ct = d(cx-s)$ is a multiple of $d$. Since $c$ and $d$ are relatively prime, this implies that $t$ is a multiple of $d$, say $t = bd$. Since $t \in \nsgtwo$, we conclude that $b \in \frac{1}{d}\nsgtwo$.  Similarly, we can write $s = ac$ for some $a$ and so $a \in \frac{1}{c}\nsg$.  

Substituting $t=bd$ and $s=ac$ into \eqref{eq:sums}, we obtain
\[
cdx = dac + cbd = cd(a+b).
\]
By cancellation, we obtain $x = a + b$ with $a \in \frac{1}{c}\nsg$ and $b \in \frac{1}{d}\nsgtwo$, as desired.  
\end{proof}

Given the ease of proving Theorem~\ref{thm:sums}, it is surprisingly more difficult when the denominators do have a common factor.  In a follow-up to this current paper, we will translate the quotient operation into a geometric setting, which will allow us to generalize Theorem~\ref{thm:sums} to drop the ``coprime denominators'' hypothesis.  Intriguingly, the translation can cause a large blow-up in the numbers, e.g.,
\[
\frac{\langle 11,13\rangle}{2} + \frac{\langle 17,19\rangle}{2}
= \frac{\langle 2416656, 2894591, 3441983, 3869571  \rangle}{25357536}.
\]
Based on experimentation, this blow-up seems necessary.

\section{Some families of numerical semigroups with full quotient rank}
\label{sec:fullquotientrank}

In this section, we produce two families of numerical semigroups:\ those in the first have embedding dimension $k+1$ but are not $k$-quotients, so in particular have full quotient rank (Theorem~\ref{thm:noquotient}); and those in the second are not even \emph{intersections} of $k$-quotients  (Theorem~\ref{thm:nointersection}).  

\begin{thm} \label{thm:noquotient}
Given a positive integer $k$, let $a\ge 2^k$ be an integer. Define $a_i=2a+2^i$ for $i=0,1,\dots,k$. Then the numerical semigroup
\[\nsg=\langle a_0,a_1,\ldots,a_k\rangle\]
is not a $k$-quotient.
\end{thm}

\begin{proof}
For $1\le j\le 2^k-1$, let $b_j=\omega(j)a+j$, 
where $\omega(j)$ is the number of 1's in the binary representation of $j$. We first prove that, if  $\nsgtwo$ is \emph{any} $k$-quotient the contains $a_0,\ldots,a_k$ (so $\nsgtwo=\nsg$ will be an example), then there exists $j$ ($1\le j\le 2^k-1$) such that $b_j\in \nsgtwo$. Indeed, we apply Theorem~\ref{thm:necessary}. We know that there exists a nonempty $I\subseteq\{0,1,\ldots,k\}$ such that $a_I/2\in \nsgtwo$. If $0\in I$, then $a_I$ is odd and $a_I/2$ is not an integer, so we know $I\subseteq\{1,\ldots,k\}$. Let
\[j=\sum_{i\in I}2^{i-1}.\]
We have that $1\le j\le 2^k-1$, and
\[a_I/2=\sum_{i\in I}\left(2a+2^i\right)/2 = \abs{I}a + \sum_{i\in I}2^{i-1} = \omega(j)a+j=b_j,\]
so $b_j\in \nsgtwo$.

Now we this apply to $\nsgtwo=\nsg$. Seeking a contradiction, suppose $\nsg$ is a $k$-quotient, and therefore we have some  $b_j\in \nsg$, that is, $b_j=\sum_{i=0}^ka_ix_i$ with $x_i\in\NN$. Examining this sum modulo $a$, and noting that $b_j=j\pmod a$ and $a_i=2^i\pmod a$, we see that
\[\sum_{i=0}^k x_i\ge \omega(j).\] But a sum of $\omega(j)$ generators of $\nsg$ is too large:
\[\omega(j)a+j=b_j\ge \omega(j)\cdot a_0 = \omega(j)(2a+1)\ge \omega(j)a+a\ge \omega(j)a+2^k,\]
a contradiction. Therefore $b_j\notin \nsg$, and so $\nsg$ cannot be a $k$-quotient.
\end{proof}

\begin{thm} \label{thm:nointersection}
Given a positive integer $k\ge 2$, let $a\ge k2^k$ be an integer. As before, define $a_i=2a+2^i$ and $b_j=\omega(j)a+j$, 
where $\omega(j)$ is the number of 1's in the binary representation of $j$. Let $N=(2k+1)a$. Then
\[\nsg=\langle a_0,a_1,\ldots,a_k,N-b_1,N-b_2,\ldots,N-b_{2^k-1}\rangle\]
cannot be written as an intersection of $k$-quotients.
\end{thm}

\begin{proof}
Suppose, seeking a contradiction, that $\nsg=\bigcap_{\ell=1}^p \nsg_\ell$, where the $\nsg_\ell$ are $k$-quotients. 
Each $\nsg_\ell$ must contain $a_0,a_1,\ldots,a_k$, and we noted in the proof of Theorem~\ref{thm:noquotient} that this implies that $\nsg_\ell$ must contain $b_j$ for some $j$. But then $\nsg_\ell$ contains both $b_j$ and $N-b_j$, and so additive closure implies that it contains $N$.  This means $N \in \bigcap_{\ell=1}^p \nsg_\ell = \nsg$. Let
\begin{equation}\label{eq:sumforN}
N=\sum_{i=1}^k a_ix_i + \sum_{j=1}^{2^k-1}(N-b_j)y_j,
\end{equation}
where $x_i,y_j \in \NN$.
We break into three cases.
\begin{itemize}
\item 
If $\sum_{j}y_j\ge 2$, then~\eqref{eq:sumforN} would be too large, as for some $j_1,j_2$,
\begin{align*}(2k+1)a
= N
&\ge (N-b_{j_1})+(N-b_{j_2})\\
&= 2N-(\omega(j_1)+\omega(j_2))a -(j_1+j_2)\\
&> 2\cdot (2k+1)a -2ka-2\cdot 2^k\\
&= (2k+2)a-2^{k+1},
\end{align*}
which is impossible since $a\ge 2^{k+1}$.

\item 
If $\sum_{j}y_j=1$, then~\eqref{eq:sumforN} uses exactly one $N-b_j$.  But then $N=(N-b_j)+b_j$ implies that $b_j\in \langle a_0,a_1\ldots,a_k\rangle$, which we saw was impossible in the proof of Theorem~\ref{thm:noquotient} since $a\ge 2^k$.

\item 
If $\sum_{j}y_j=0$, then $N = \sum_{i} a_ix_i$. If $\sum_i x_i \le k$, then
\[(2k+1)a=N\le k(2a+2^k),\]
which is impossible since $a>k2^k$. On the other hand, if $\sum_i x_i > k$, then
\[(2k+1)a=N\ge (k+1)(2a+1)>(2k+1)a,\]
which is also impossible.
\end{itemize}
In each case, we obtain a contradiction.  
\end{proof}

\section{How often do numerical semigroups have full quotient rank?}
\label{sec:randomsgps}

In this section, we consider the question ``how likely is a randomly selected numerical semigroup to have full quotient rank?''  We consider two sampling methods.  The first is the ``box'' method, wherein a fixed number of generators are selected uniformly and independently from an interval $[1,M]$.  Numerical semigroups selected under this model have high probability (i.e., approaching 1 as $M \to \infty$) of having full quotient rank.  

\begin{thm} \label{thm:numericalbox}
Fix a positive integer $n$. If $\nsg = \langle a_1, \dots, a_n \rangle$ where $a_1, \ldots, a_n \in [M]$ are uniformly and independently chosen, then the probability that $\nsg$ has full quotient rank tends to 1 as $M \to \infty$.  More precisely, this probability is $1 - O(M^{-\frac{1}{n}})$.  
\end{thm}

\begin{proof}
By Corollary~\ref{cor:necessary}, it suffices to bound the probability that there exists $I \subseteq [n]$ such that $a_I \in \langle a_j:\ j\notin I\rangle$.  Let $A$ be this event, and let $B$ be the event that $a_i \leq M^{\frac{n-1}{n}}$ for some $i$. We will use that
\[
\Pr(A)
= \Pr(B)\Pr(A \mid B) + \Pr(B^c)\Pr(A \mid B^c)
\leq \Pr(B) + \Pr(A \mid B^c).
\]
For the first term, the union bound gives us that
\[
\Pr(B)
\leq n \left( \frac{M^{\frac{n-1}{n}}}{M} \right)
= \frac{n}{M^{\frac{1}{n}}}.
\]
For the second term, fix a nontrivial subset $I \subsetneq [n]$ and $b_i \in \NN$ for $i \notin I$. If $b_i > nM^{\frac{1}{n}}$ for some $i \notin I$, then since every $a_i$ is greater than $M^{\frac{n-1}{n}}$, we have
\[
\sum_{j \notin I} b_ja_j
\geq b_ia_i
> \left( nM^{\frac{1}{n}} \right) M^\frac{n-1}{n}
= nM.
\]
But $a_I$ cannot be this large because it is the sum of at most $n-1$ integers that are each at most $M$. So we need only consider $b_i \le nM^{\frac{1}{n}}$.  Letting $i^\ast = \min(I)$ and $m = nM^{\frac{1}{n}}$, 
\begin{align*}
\Pr(A \mid B^c)
&\le \sum_{\substack{I \subsetneq [n] \\ I \ne \emptyset}} \sum_{\substack{b_j \le m \\ j \notin I}} \Pr \bigg( \sum_{i \in I}a_i =  \sum_{i \notin I}b_ia_i \biggm\vert a_1, \ldots, a_n > M^{\frac{n}{n-1}} \bigg)
\\
& = \sum_{\substack{I \subsetneq [n] \\ I \ne \emptyset}} \sum_{\substack{b_j \le m \\ j \notin I}} \Pr \bigg( a_{i^\ast} = \sum_{i \notin I}b_ia_i - \!\!\! \sum_{i \in I \setminus \{i^\ast\}} \!\!\! a_i \biggm\vert a_1, \ldots, a_n > M^{\frac{n}{n-1}} \bigg)
\\
& \le \sum_{\substack{I \subsetneq [n] \\ I \ne \emptyset}} \sum_{\substack{b_j \le m \\ j \notin I}} \frac{1}{M-M^{\frac{n-1}{n}}}
\le \frac{\left( 2^n - 2 \right) \left( nM^{\frac{1}{n}} \right)^{n-1}}
{M-M^{\frac{n-1}{n}}}
= \frac{\left( 2^n - 2 \right)n^{n-1}}{M^{\frac{1}{n}}-1},
\end{align*}
where the second inequality comes from the fact that for any choice of the $a_i$ with $i \neq i^\ast$, there is at most one choice of $a_i^\ast$ that makes the linear equation hold.  Thus,
\[
\Pr(A) \leq \frac{\left( 2^n - 2 \right)n^{n-1}}{M^{\frac{1}{n}}-1} + \frac{n}{M^{\frac{1}{n}}-1} = O(M^{-\frac{1}{n}}),
\]
which completes the proof.  
\end{proof}

\begin{remark}\label{rem:numericalbox}
The ``minimally generated'' and ``finite complement'' conditions, which are often imposed on numerical semigroups, do not affect Theorem~\ref{thm:numericalbox}.  
Indeed, under this ``box'' probability model, the chosen generators $a_1, \dots, a_n$ need not form a minimal generating set.  Since the quotient rank is at most the embedding dimension, the (asymptotically rare) event that the rank of $\nsg$ is less than $n$ contains the event that the chosen generating set is not minimal.  
Additionally, the probability that $a_1,\ldots, a_n$ are relatively prime approaches the positive constant $1/\zeta(n)$ by~\cite{Nym}, where $\zeta(n)$ is the Reimann zeta function $\sum_{i=1}^{\infty}1/i^n$.  Therefore, even if one restricts to those $a_1,\ldots, a_n$ that are relatively prime, the conditional probability that the quotient rank of the resulting numerical semigroup is less than $n$ still tends to 0.
\end{remark}

Under the second model, a numerical semigroup $\nsg$ is selected uniformly at random from among the (finitely many) with fixed smallest generator $m$ and number of gaps~$g$.  Such numerical semigroups have high probability (i.e., tending to 1 as $g \to \infty$) of having embedding dimension $m$ (such numerical semigroups are said to have \emph{maximal embedding dimension}).  We prove that maximal embedding dimension numerical semigroups never have full quotient rank, illustrating a stark contrast in asymptotic behavior to the first model.  

We first recall a characterization of quotient rank 2 numerical semigroups, which appears in~\cite{numerical} as a characterization of proportionally modular numerical semigroups in the case $\gcd(\nsg) = 1$.  Our statement here is more general, thanks to Remark~\ref{rem:relprime}.  

\begin{thm}\label{thm:pmcriterion}
A numerical semigroup $\nsg$ with $\gcd(\nsg) = D$ has quotient rank 2 if and only if there exists an ordering $b_1, \ldots, b_n$ of its minimal generators such that:
\begin{enumerate}[(a)]
\item $\gcd(b_i, b_{i+1}) = D$ for $1 \leq i \leq n-1$; and
\item $b_{i-1} + b_{i+1}$ is divisible by $b_i$ for $2 \leq i \leq n-1$. 
\end{enumerate}
\end{thm}

\begin{lemma}\label{lem:plusminusone}
For any $a, b, m \ge 1$, the numerical semigroup $\nsg = \langle m,am-1, bm+1 \rangle$ is a 2-quotient. 
\end{lemma}

\begin{proof}
If $\mathsf e(\nsg) \le 2$, then $\nsg$ is clearly a 2-quotient. Otherwise, letting $b_1 = am-1$, $b_2 = m$, and $b_3 = bm+1$, it is clear that $\gcd(b_1, b_2) = \gcd(b_2, b_3) = 1$ and that $b_2 \mid  (b_1 + b_3)$.  As such, $\nsg$ is a 2-quotient by Theorem \ref{thm:pmcriterion}.  
\end{proof}

\begin{thm}\label{thm:maxembdim}
If $m = \min(\nsg \setminus \{0\})$, then $\nsg$ is an $(m-1)$-quotient.  In particular, if $\mathsf e(\nsg) = m$, then $\nsg$ does not have full quotient rank.  
\end{thm}

\begin{proof}
If $\nsg=\frac{\nsg}{1}$ has embedding dimension less than $m$, then the proof is immediate. If not, then $\nsg$ has $m$ minimal generators, and so they must all have distinct residues modulo $m$. That is,
\[ \nsg = \langle m, b_1m+1, \dots, b_{k-1}m + (m-1) \rangle \]
for some positive integers $b_1, \dots, b_{m-1}$. Write $\nsg = \nsg_1 + \nsg_2$ where
\[ \nsg_1 = \langle m, b_1m + 1, b_{k-1}m + (m-1) \rangle, \: \nsg_2 = \langle b_2m + 2, \dots, b_{m-2}m + (m-2) \rangle. \]
Now by Lemma~\ref{lem:plusminusone}, $\nsg_1$ as a 2-quotient, and $\nsg_2 = \frac{\nsg_2}{1}$ is trivially an $(m-3)$-quotient.  Since $1$ is coprime to every integer, Theorem~\ref{thm:sums} implies $\nsg$ is an $(m-1)$-quotient. 
\end{proof}


\section*{Acknowledgements}
Tristram Bogart was supported by internal research grant INV-2020-105-2076 from the Faculty of Sciences of the Universidad de los Andes. 


\end{document}